\documentclass[reqno,11pt]{amsart}
\usepackage{times}
\usepackage{amsfonts,amsmath,amssymb}
\usepackage{mathrsfs,enumerate}
\usepackage[latin1]{inputenc}
\usepackage[T1]{fontenc}
\numberwithin{equation}{section}
\newtheorem{theorem}{Theorem}[section]
\newtheorem{lemma}[theorem]{Lemma}

\allowdisplaybreaks
\newcommand{\R}{\mathbb{R}}

\newcommand{\norm}[1]{\left\Vert#1\right\Vert}
\newcommand{\abs}[1]{\left\vert#1\right\vert}
\newcommand{\set}[1]{\left\{#1\right\}}
\newcommand{\para}[1]{\left(#1\right)}

\newcommand{\p}{\partial}

\begin{document}
\title[Determining the implied volatility in the Dupire equation for vanilla European call options]{Determining the implied volatility in the Dupire equation for vanilla European call options}
\author{M. Bellassoued, R. Brummelhuis, M. Cristofol, E. Soccorsi}
\address{Mourad Bellassoued, University of  Carthage, Faculty of Sciences of Bizerte, Department of Mathematics, 7021 Jarzouna Bizerte, Tunisia}
\email{mourad.bellassoued@fsb.rnu.tn}
\address{Raymond Brummelhuis, LMR, EA 4535, Universit\'e de Reims, France}
\email{raymondus.brummelhuis@univ-reims.fr}
\address{Michel Cristofol, LATP-CNRS, UMR 7353, Aix-Marseille Universit\'e, France}
\email{cristo@cmi.univ-mrs.fr}
\address{Eric Soccorsi, CPT-CNRS, UMR 7332, Aix-Marseille Universit\'e, France}
\email{eric.soccorsi@univ-amu.fr}
\date{}
\begin{abstract}
The Black-Scholes model gives vanilla Europen call option prices as a function of the volatility. 
We prove Lipschitz stability in the inverse problem of determining the implied volatility, which is a function of the underlying asset, from a collection of quoted option prices with different strikes.
\end{abstract}
\maketitle

\section{Introduction and main result}
\label{sec-intro}

An option is a contract to buy or sell a specific financial product known as the option's underlying asset. Its primary uses are speculation and hedging in the sense that this is a cheap way of either making important returns by exposing a porfolio to a large amount of risk, or reducing the risk arising from unanticipated changes in the underlying price. The commonest model of the asset price $S$ at time $t$ is given by the stochastic differential equation
\begin{equation}
\label{rw}
dS = \mu S dt + \sigma S dW,
\end{equation}
where the randomness, which is a feature of asset prices, is contained in the Wiener process $W$.
Here $\mu$ is a measure of the average rate of growth of the asset price and the coefficient $\sigma$, called {\it volatility}, measures the standard deviation of the returns.

The holder of a vanilla European call option, which is the simplest financial option, has the right (but not the obligation) to buy the underlying asset $S$ at a prescribed time $T$ in the future, known as the expiry date, for a prescribed amount $K$, called the exercise or strike price. In an ideal financial market, the price at time $t$ of this option, $v (S, t)$, is the solution to the celebrated Black-Scholes equation (see  \cite{BS})
\begin{equation}\label{1.1}
 \frac{\p v}{\p t}+\frac{1}{2}S^2\sigma^2\frac{\p^2v}{\p S^2}+(r-q)S\frac{\p v}{\p S}-rv=0.
\end{equation}
where $r$ is the risk-free interest rate and $q$ is the continuous dividend yield. We notice that (\ref{1.1}) does not contain the growth parameter $\mu$, showing that the value of an option is independent on how rapidly an asset grows. The only coefficient in the random walk (\ref{rw}) governing the asset price that affects the option price is the volatility $\sigma$.

In practice, both $r$ and $q$ are known and measurable quantities. We assume in this text that they are constant, with $r>0$. In contrast to $r$ and $q$ the volatility $\sigma$ is an unknown and unobservable coefficient of the above described system. In the general case $\sigma$ depends on $S$ and $t$ but we suppose in the sequel that $\sigma$ is a function of $S$ only, fulfilling
\begin{equation}
\label{1.2}
0 < \sigma_{\min} \leq \sigma(S) \leq \sigma_{\max} < \infty,\ S \geq 0.
\end{equation}

Since (\ref{1.1}) is a backward parabolic equation and the value $v(S,t)$ should be unique, suitable final and boundary conditions have to be prescribed in order to make the solution 
to the Black-Scholes equation unique. 
Namely, the final condition of a vanilla European call option is just its payoff at time $T$:
\begin{equation}\label{1.3}
 v(S,T)=\max(S-K,0),\ S\geq 0.
\end{equation}
Further, as $S$ belongs to $(0,+\infty)$, we need to impose a boundary condition at $S=0$. 
This can be done by plugging $S = 0$ in (\ref{1.1}), getting
$$
\frac{\p v}{\p t}(0,t)=rv(0,t),\ t \in (0,T),
$$
and then combining the above equation with the identity $v(0,T)=0$, arising from (\ref{1.3}). This entails
\begin{equation}\label{1.5}
v(0,t)=0,\ t \in (0,T).
\end{equation}
It is well known that (\ref{1.1}) and (\ref{1.3})-(\ref{1.5}) admits a unique solution
$$
v \in \mathcal{C}^0((0,\infty)\times[0, T])\cap \mathcal{C}^1((0,\infty)\times (0, T]).
$$
If $S$ becomes arbitrarily large, then it becomes ever more likely that the option will be exercised, whatever the magnitude of the strike price is. Thus, as
$S \to\infty$, the value of the option is determined by the price $S$ of the asset minus the amount of money $M(t)$ paid at time $t \in (0,T)$ to exercise the option at future time $T$. Since money $M(t)$ in a bank with constant interest rate $r$ grows exponentially according to $M'(t)=r M(t)$, the value at time $t$ of a payoff $K$ at $t=T$ is $M(t)=K e^{-r(T-t)}$. Therefore, we have
$$
\lim_{S\to\infty}\para{v(S,t)-Ke^{-r(T-t)}}=0,\ t \in (0,T).
$$

The Black-Scholes model defined by (\ref{1.1}) and (\ref{1.3})-(\ref{1.5}) gives option prices as a function of $\sigma$. The volatility can be estimated from historical data, but this does not accurately predicts the future volatility required by the model in practice. 
However, since option prices are quoted in the market, it is commonly admitted that the volatility is fully determined from the market. In mathematical language this can be translated by saying that the volatility can be identified from 
option prices, or that there is a one-to-one correspondence between the volatility and the option quotes. This volatility coefficient, derived from the quoted price for a single option, is called the {\it implied volatility}.

Having said that, we fix $S^*>0$, $t^* \in (0,T)$, and pick a subinterval $I^*$ of $(0,+\infty)$. Then the inverse problem of option pricing is to determine the pair of functions
$v(S,t)$ and $\sigma(S)$ fulfilling (\ref{1.1})--(\ref{1.5}) from the data
\begin{equation}\label{1.6}
v(S^*,t^*;K,T,\sigma)\equiv \phi_\sigma^*(K),\ K \in I^*.
\end{equation}
The main result of this article, stated below, claims that the implied volatility $K \mapsto \sigma(K)$ can be actually restored stably on any interval $I_1 \Supset I^*$ from the collection of simultaneous option quotes $\phi_\sigma^*(K)$ with different strikes $K \in I_1^*$, provided $\sigma$ is known in $(0,+\infty) \setminus I^*$. 

\begin{theorem}\label{T1.1}
Let $\sigma_j \in L^{\infty}(0,+\infty)$, $j=1,2$, fulfill (\ref{1.2}), let 
$v_j$ denote the $\mathcal{C}((0,\infty)\times[0, T])\cap \mathcal{C}^1((0,\infty)\times (0, T])$-solution to (\ref{1.1}) and (\ref{1.3})-(\ref{1.5}), 
where $\sigma_j$ is substituted for $\sigma$ in (\ref{1.1}), and let $\phi_{\sigma_j}$ be the final data (\ref{1.6}) associated to $v_j$. Assume that $\sigma_1$ and $\sigma_2$ coincide outside some bounded subinterval $I^* \subset (0,\infty)$. Then for any interval $I^*_1\Supset I^*$ there is a constant $C>0$, depending only on $\sigma_{max}$, $I^*$ and $I_1^*$, such that we have
$$
\norm{\sigma_1-\sigma_2}_{L^2(I^*_1)}\leq C\norm{\phi_{\sigma_1}-\phi_{\sigma_2}}_{H^2(I^*_1)}.
$$
\end{theorem}
\section{The Dupire equation}
In this section we rewrite the inverse boundary value problem introduced in \S \ref{sec-intro} into an equivalent parameter identification problem associated to some suitable forward parabolic PDE.

For all $K>0$ and $T>0$, we note$v(S,t;K,T)$ the $\mathcal{C}((0,\infty)\times[0, T])\cap \mathcal{C}^1((0,\infty)\times (0, T])$ solution to (\ref{1.1}) and (\ref{1.3})-(\ref{1.5}).
As first observed in \cite{Dup} and rigorously justified in \cite{Bou-Isa}, for $S>0$ and $t \in (0,T)$ fixed, the option price $v(S,t;K,T) \equiv u(K,T)$,
as a function of the expiry date $T$ and the strike price $K$, satisfies the 
equation dual to (\ref{1.1}),
\begin{equation}\label{2.1}
\frac{\p u}{\p T}-\frac{1}{2}K^2\sigma^2(K)\frac{\p^2 u}{\p K^2}+(r-q) K\frac{\p u}{\p K}+qu=0, \quad T\in (t,\infty),\quad K\in (0,\infty),
\end{equation}
with the initial condition
\begin{equation}\label{2.2}
u(K, t) = \max(S - K,0),\ K > 0,
\end{equation}
and the boundary conditions
\begin{equation}\label{2.3}
u(0, T) = Se^{-q(T-t)},\quad \lim_{K\to\infty} u(K, T) = 0,\ T \in (t,\infty).
\end{equation}
In light of the inverse problem described in \S \ref{sec-intro} we thus aim to determine $\sigma(K)$ from the set of option quotes
$\{ u(K,T;S^*,t^*),\ K \in I_1^* \}$, for some fixed values of $T>0$, $S^*>0$ and $t^* \in (0,T)$.

Actually, the following logarithmic substitution
$$
y=\ln\frac{K}{S},\ \tau=T-t,\ w(y,\tau)=u(K,T)e^{q\tau},\ a(y)=\frac{1}{2}\sigma^2(K),
$$
transforms (\ref{2.1})--(\ref{2.3}) into the problem
$$
\para{\p_\tau-\mathcal{L}_a}w=0,\ \tau>0,\ y\in\R,
$$
where
$$
\mathcal{L}_aw=\mathcal{L}_a(y,\p_y)w\equiv a(y)w_{yy}-\para{a(y)+r-q}w_y,
$$
with the initial condition
$$
w(y,0)=S\max(1-e^y,0),\ y\in\R,
$$
and the boundary conditions
$$
\lim_{y\to-\infty} w(\tau,y) =S,\quad \lim_{y\to+\infty}w(\tau,y) = 0,\ \tau \in (0, \infty).
$$
Therefore, $S^*>0$ and $t^*>0$ being arbitrarily fixed, the above mentioned inverse problem may be rephrased as to whether $a(y)=\frac{1}{2}\sigma^2(S^*e^y)$ can be retrieved for $y\in\Omega_1=\set{ \ln(K/S^*),\ K \in I_1^*}$
from the knowledge of $\set{w(y,\tau^*),\ y\in\Omega_1}$, where $\tau^*=T-t^*$ and $w$ is the solution
to the system
\begin{equation}\label{2.6}
\left\{\begin{array}{lll}
\para{\p_\tau-\mathcal{L}_a}w=0, & y\in\R,\ \tau\in (0,\infty),\cr
w(y,0)=S^*\max(1-e^y,0), & y\in\R,\cr
\displaystyle\lim_{y\to-\infty} w(\tau,y) =S^*,\quad \displaystyle\lim_{y\to+\infty}w(\tau,y) = 0,& \tau \in (0, \infty).
\end{array}
\right.
\end{equation}

\section{Carleman estimate for the Dupire equation}
In this subsection we recall some parabolic Carleman estimate, which is useful to the proof of Theorem \ref{T1.1}. To this purpose we first define the set
$$
\Omega=\set{y=\ln\frac{K}{S^*},\,K\in I^*},
$$
fulfilling $\Omega\Subset \Omega_1$ since $I^*\Subset I_1^*$. Next we pick a non-empty interval $\omega_2\subset \Omega_1\backslash\Omega$ together with two subsets
\begin{equation}\label{3.1}
\omega\subset\omega_1\subset\omega_2.
\end{equation}
As shown in Appendix, there exists a function $\psi_0\in \mathcal{C}^2(\Omega_1)$ obeying
\begin{equation}\label{3.2}
\psi_0>0\ \textrm{in}\ \Omega_{1},\ \psi'_0\neq 0\ \textrm{in}\ \Omega_{1}\backslash \omega\ \textrm{and}\ \psi_0=0\ \textrm{in}\ \p \Omega_{1}.
\end{equation}
Since $\psi_0(y)>0$ for all $y\in \overline{\Omega}$, there is a constant $\delta>0$ such that
\begin{equation}\label{3.3}
\psi_0(y)\geq 2\delta,\ y\in \overline{\Omega}.
\end{equation}
Moreover, as we have $\psi_0=0$ on $\p \Omega_1$, we may find an interval $\Omega_2$ satisfying $\Omega\subset \Omega_2\subset \Omega_1$ for which
\begin{equation}\label{3.4}
\psi_0(y)\leq \delta,\ y\in \Omega_1\backslash \Omega_2.
\end{equation}
Put $T=2\tau^*$, $\ell(\tau)=\tau(T-\tau)$ and $\psi(y)=\psi_0(y)+\overline{\psi}$ with $\overline{\psi}=2\norm{\psi_0}_\infty$. For all $\lambda > 0$, we define two weight functions
$$
\varphi(y,\tau)=\frac{e^{\lambda\psi(y)}}{\ell(\tau)},\ y\in \Omega_1,\ \tau\in (0,T),
$$
and
$$
\eta(y,\tau)=\frac{e^{\lambda\psi(y)}-e^{2\lambda\overline{\psi}}}{\ell(\tau)},\ y\in \Omega_1,\ \tau\in (0,T).
$$
Finally we introduce the following functional space
$$
H^{1,2}(Q)=\set{z\in L^2(Q),\,z_\tau,\,z_y,\,z_{yy}\in L^2(Q)},\quad Q=\Omega_1\times(0,T),
$$
endowed with the norm
$$
\norm{z}_{H^{1,2}(Q)}=\norm{z}_{L^2(Q)}+\norm{z_\tau}_{L^2(Q)}+\norm{z_y}_{L^2(Q)}+\norm{z_{yy}}_{L^2(Q)},
$$
and recall the following parabolic Carleman estimate proved in \cite{[FI],[IY2]}:
\begin{lemma}\label{L3.1} Let $a\in \mathcal{C}^2(\Omega_1)$.
Then there exists $\lambda_0 > 0$ such that for all $\lambda > \lambda_0$ we may find two constants $s_0> 0$ and $C>0$ satisfying
\begin{multline}
\int_{0}^T\!\!\!\int_{\Omega_{1}} \para{s\ell^{-1}(\tau)\abs{z_y}^2+s^3\ell^{-1}(\tau)\abs{z}^2+s^{-1}\ell(\tau)\abs{z_\tau}^2}e^{2s\eta}dyd\tau\cr
 \leq C\Bigg(\int_{0}^T\!\!\!\int_{\Omega_{1}}
 \abs{\para{\p_\tau-\mathcal{L}_a}z}^2e^{2s\eta}dyd\tau +s^3\int_{0}^T\!\!\!\int_{\omega}\ell^{-3}(\tau)\abs{z}^2 e^{2s\eta}dyd\tau\Bigg),
\end{multline}
for every $s \geq s_0$ and every $z \in H^{1,2}(Q)$ obeying $z(y,\tau)=0$ for $(y,\tau) \in \p\Omega_1 \times (0,T)$. Here $C$ depends continuously on $\lambda$, and is independent of $s$.
\end{lemma}
\section{Parabolic interior estimates}
In this section we derive several preliminary PDE estimates
which are essential to the analysis of the inverse problem.

Let $w$ be solution to the initial value problem
\begin{equation}\label{4.1}
\para{\p_\tau-\mathcal{L}_a(y,\p)}w=F,\quad \textrm{for all},\ y\in\R,\ t\in (0,T),
\end{equation}
with the initial data
\begin{equation}\label{4.2}
w(y,0)=0,\ y\in\R,
\end{equation}
and boundary condition
$$
\lim_{y\to\pm\infty}w(y,\tau)=0,\ \tau\in (0,T).
$$
Several estimates, required in the derivation of the stability inequality stated in Theorem \ref{T1.1}, are collected in a succession of four lemmas.
\begin{lemma}\label{L4.1}
For each $T>0$ there is a constant $C=C(T)>0$ such that
$$
\norm{w(\tau,\cdot)}^2_{L^2(\R)}+\int_0^\tau\norm{w_y(s,\cdot)}_{L^2(\R)}^2ds\leq C\int_0^\tau\norm{F(s,\cdot)}_{L^2(\R)}^2ds.
$$
\end{lemma}
The proof of this classical result, which is by means of Gronwall Theorem, can be found in \cite{Evans}.

Further we recall from \cite{Evans}[chap.7, Thm 5] the:
\begin{lemma}\label{L4.2} Let $I\subset\R$ be a bounded interval, let $F\in L^2(0,T;L^2(I))$ and let $v\in H^{1,2}(I\times(0,T))$ be solution to
$$
\left\{\begin{array}{lll}
\para{\p_\tau-\mathcal{L}_a(y,\p)}v=F, & y\in I,\ t\in (0,T)\cr
v(y,0)=0,&  y\in I,\cr
v(y,\tau)=0 & y\in\p I,\ t\in (0,T).
\end{array}
\right.
$$
Then it holds true that
$$
\norm{v}_{H^{1,2}(0,T)\times I}\leq C\norm{F}_{L^2(0,T;L^2(I))}.
$$
Moreover if $F_\tau\in L^2(0,T;L^2(I))$ we have in addition
\begin{equation}\label{4.7}
  \sup_{0\leq\tau\leq T}\norm{v(\cdot,\tau)}_{H^2(I)}\leq C\norm{F}_{H^1(0,T;L^2(I))}.
\end{equation}
\end{lemma}
\begin{lemma}\label{L4.3}
Fix $J_2\subset\R$. Then there is a constant $C>0$ such that for every $w\in H^{1,2}(J_2\times(0,T))$ obeying (\ref{4.1})-(\ref{4.2}) we have:
\begin{equation}\label{4.8}
\norm{w}_{H^{1,2}(J_2\times(0,T))}\leq C\norm{F}_{L^2(0,T;L^2(\R))}.
\end{equation}
Furthermore, if $F(y,\tau)=0$ for all $(y,\tau)\in J_2\times(0,T)$ then
\begin{equation}\label{4.9}
\norm{w_{\tau,y}}_{L^2(J_1\times(0,T))}\leq C\norm{F}_{L^2(\R\times (0,T))},
\end{equation}
for any $J_1\Subset J_2$ such that $\mathrm{dist}(J_2,\R\backslash J_1)>0$.
\end{lemma}
\begin{proof}
Let $\phi$ be a cutoff function supported in some interval $\widetilde{J}_2\Supset J_2$ and fulfilling $\phi(y)=1$ in $J_2$. Then $v=\phi w$ is solution to the following system
$$
\left\{\begin{array}{lll}
\para{\p_\tau-\mathcal{L}_a(y,\p)}v=\phi F+Q_1(y,\p_y)w, & y\in \widetilde{J}_2,\ t\in (0,T),\cr
v(y,0)=0, & y\in \widetilde{J}_2,\cr
v(y,\tau)=0, & y\in\p \widetilde{J}_2,\ t\in (0,T),
\end{array}
\right.
$$
where $Q_1$ denotes some first order partial differential operator. By applying Lemma \ref{L4.2}, we get
$$
\norm{v}_{H^{1,2}(\widetilde{J}_2\times(0,T))}\leq C\para{\norm{F}_{L^2(\R\times(0,T))}+\norm{w}_{L^2(\R\times(0,T))}+\norm{w_y}_{L^2(\R\times(0,T))}},
$$
so (\ref{4.8}) follows directly from this and Lemma \ref{L4.1}.\\
To prove (\ref{4.9}) we put $u=(\xi w)_y$ where $\xi\in \mathcal{C}_0^\infty(J_2)$ satisfies $\xi(y)=1$ for $y \in J_1$. Since $\xi F=0$ the function $u$ is solution to the system
$$
\left\{\begin{array}{lll}
\para{\p_\tau-\mathcal{L}_a(y,\p)}u=Q_2(y,\p_y)w, & y\in J_2,\ t\in (0,T),\cr
u(y,0)=0, & y\in J_2,\cr
u(y,\tau)=0 & y\in\p J_2,\ t\in (0,T),
\end{array}
\right.
$$
where $Q_2$ is a second order differential operator. Applying Lemma \ref{L4.2}, we obtain
$$
\norm{u_\tau}_{L^2(J_2\times(0,T))}\leq C\norm{w}_{L^2(0,T;H^2(J_2))}\leq C\norm{F}_{L^2(\R\times(0,T))},
$$
which proves the result.
\end{proof}
\begin{lemma}\label{L4.4}
Let $\Omega \Subset \Omega_1\subset\R$ and assume that $F(y,\tau)=0$ for every $(y,\tau)\in (\R\backslash\Omega)\times(0,T)$. Then there is a constant $C>0$ such that
$$
\sup_{0\leq t\leq T}\norm{w(\tau,\cdot)}_{H^2(\Omega_1)}\leq \norm{F}_{H^1(0,T;L^2(\Omega))},
$$
for every solution $w\in H^{1,2}(\Omega_1\times(0,T))$ to (\ref{4.1})-(\ref{4.2}).
\end{lemma}
\begin{proof}
Pick $\Omega_2\Supset\Omega_1$ such that $\textrm{dist}(\Omega_1,\R\backslash \Omega_2)>0$,
and let $\phi\in \mathcal{C}_0^\infty(\Omega_2)$ satisfy $\phi(y)=1$ for $y \in \Omega_1$. Then $v=\phi w$ is solution to the system
$$
\left\{\begin{array}{lll}
\para{\p_\tau-\mathcal{L}_a(y,\p)}v=\phi F+Q_1(y,\p_y)w, & y\in \Omega_2,\ t\in (0,T),\cr
v(y,0)=0, & y\in \Omega_2,\cr
v(y,\tau)=0, & y\in\p \Omega_2,\ t\in (0,T),
\end{array}
\right.
$$
for some first order operator $Q_1$ which is supported in $J_1=\Omega_2\backslash \Omega_1$. As a consequence we have
$$
\sup_{0\leq \tau\leq T}\norm{v(\cdot,\tau)}_{H^2(\Omega_2)}\leq C\para{\norm{F}_{H^1(0,T;L^2(\R))}+\norm{w_{\tau,y}}_{L^2((0,T)\times J_1)}},
$$
directly from Lemma \ref{L4.2}. Now the result follows from this and (\ref{4.9}) since
$F(\cdot,\tau)$ vanishes in a neighborhood of $J_1$ for each $\tau\in (0,T)$.
\end{proof}
\section{Stability estimate for the linearized inverse problem}
As $w=w_{a_1}-w_{a_2}$ is solution to the linearized system
\begin{equation}\label{5.1}
\left\{
\begin{array}{llll}
\para{\p_\tau-\mathcal{L}_{a_1}(y,\p_y)}w=f(y)\alpha(y,\tau),  & \textrm{in }\,\, \R\times(0,T),\cr
w(y,0)=0,& \textrm{in }\,\,\R,\cr
\lim_{y\to\pm\infty}w(y,\tau)=0,& \textrm{on } \,\, (0,T),
\end{array}
\right.
\end{equation}
with $f=a_2-a_1$ and $\alpha(y,\tau)=(w_{a_2})_{yy}-(w_{a_2})_{y}$, by (\ref{2.6}), 
we now examine the inverse problem of determining $f$ from $w_{|\omega_2\times(0,T)}$ and $w(y,\tau^*)$, $y\in \Omega_1$. 

To this end we first recall from \cite{Bou-Isa} that
\begin{equation}\label{5.3}
\alpha_0=\inf_{y\in\Omega}\abs{\alpha(y,\tau^*)}>0,
\end{equation}
and then establish the following:
\begin{lemma}\label{L5.2}
There exists a constant $C>0$ such that we have
\begin{equation}\label{5.4}
\norm{f}^2_{L^2(\Omega)}
\leq C\norm{w(\cdot,\tau^*)}^2_{H^2(\Omega_1)} + \int_{0}^T\!\!\!\int_{\omega_2} \abs{w}^2 dyd\tau,
\end{equation}
for any solution $w$ to (\ref{5.1}).
\end{lemma}
\begin{proof}
Let $\phi\in C_0^\infty(\Omega_1)$ be such that $\phi(y)=1$ for $y\in\Omega_2$.
Then $z = \phi w_\tau$ satisfies
$$
\para{\p_\tau-\mathcal{L}_{a_1}}z = f(y)\alpha_\tau(y,\tau)+\mathcal{Q}_1w_\tau\ \textrm{in}\ \Omega_1\times (0,T),
$$
by (\ref{5.1}), where $\mathcal{Q}_1$ is a first order partial differential operator supported in $\Omega_1 \backslash \Omega_2$. Moreover, it holds true that
\begin{equation}\label{5.6}
z(y, \tau^*) = A_2w(\tau^*,y) + f(y)\alpha( x, \tau^*) \quad \textrm{in}\, \Omega_1,
\end{equation}
for some second order partial differential operator $A_2$. Therefore, Lemma \ref{L3.1} yields
\begin{multline}\label{5.7}
C\int_{0}^T\!\!\!\int_{\Omega_1} e^{2s\eta}\para{s^3\ell^{-3}(\tau)\abs{z}^2 + s\ell^{-1}(\tau) \abs{z_y}^2 + s^{-1}\ell(\tau)\abs{z_\tau}^2 }dyd\tau\cr
\leq s^3\int_{0}^T\!\!\!\int_{\omega} e^{2s\eta}\ell^{-3}(\tau)\abs{z}^2 dyd\tau
+ \int_{0}^T\!\!\!\int_{\Omega_1}e^{2s\eta} \abs{f(y)}^2 dyd\tau+\int_{0}^T\!\!\!\int_{\Omega_1} e^{2s\eta} \abs{\mathcal{Q}_1w_\tau}^2 dyd\tau,
\end{multline}
provided $s > 0$ is large enough.\\
Further, bearing in mind that $e^{2s\eta(y,0)} = 0$ for every $y\in\Omega_1$, we deduce from
Cauchy-Schwarz and Young's inequalities that
\begin{multline}
\int_{\Omega_1} s\abs{z(y,\tau^*)}^2e^{2s\eta(y,\tau^*)}dyd\tau=\int^{\tau^*}_0 \frac{\p}{\p \tau}\para{\int_{\Omega_1}s\abs{z(y,\tau)}^2e^{2s\eta} dy}d\tau\cr
= \int_{0}^{\tau^*}\!\!\!\int_{\Omega_1}2s^2\eta_\tau e^{2s\eta}\abs{z}^2dyd\tau
+ \int_0^{\tau^*}\!\!\!\int_{\Omega_1} 2\para{ s^{-1 \slash 2} \ell(\tau)^{1 \slash 2} z_\tau}
\para{ s^{3 \slash 2} z \ell(\tau)^{-1 \slash 2} }e^{2s\eta} dyd\tau \\
\leq C\int_{0}^{\tau^*}\!\!\!\int_{\Omega_1}
s^2\ell^{-2}(\tau) e^{2s\eta} \abs{z}^2 dyd\tau
+ C\int_{0}^{\tau^*}\!\!\!\int_{\Omega_1}
\para{ \frac{1}{s} \ell(\tau) \abs{z_\tau}^2
+ s^3\abs{z}^2\ell^{-1}(\tau)}e^{2s\eta} dyd\tau.
\end{multline}
In light of (\ref{5.7}) this entails
\begin{multline}\label{5.8}
\int_{\Omega_1} s\abs{z(y,\tau^*)}^2e^{2s\eta(y,\tau^*)}dyd\tau\leq Cs^3 \int_{0}^{\tau^*}\!\!\!\int_{\omega}
e^{2s\eta} \ell^{-3}(\tau) \abs{z}^2 dyd\tau\cr
+ C\int_0^T\!\!\!\int_{\Omega_1} e^{2s\eta} \abs{f(y)}^2 dyd\tau+C\int_{0}^T\!\!\!\int_{\Omega_1} e^{2s\eta} \abs{\mathcal{Q}_1w_\tau}^2 dyd\tau,
\end{multline}
upon taking $s > 0$ sufficiently large .\\
By substituting the right hand side of (\ref{5.6}) for $z(y,\tau^*)$ in (\ref{5.8}) we thus find out that
\begin{multline}\label{5.9}
\int_{\Omega_1} s\abs{f(y)}^2 \abs{\alpha(y,\tau^*)}^2 e^{2s\eta(y,\tau^*)}dy
\leq C\int_{\Omega_1} s\abs{A_2w(y,\tau^*)}^2e^{2s\eta(y,\tau^*)}dy\cr
+ Cs^3\int_0^T\!\!\!\int_{\omega} e^{2s\eta} \ell^{-3}(\tau)
\abs{z}^2 dyd\tau
+ C\int_0^T\!\!\!\int_{\Omega_1}e^{2s\eta} \abs{f(y)}^2 dyd\tau+C\int_{0}^T\!\!\!\int_{\Omega_1} e^{2s\eta} \abs{\mathcal{Q}_1w_\tau}^2 dyd\tau,
\end{multline}
for $s > 0$ large enough.\\
The next step of the proof is to absorb the third term in the right hand side of (\ref{5.9}) into its left side. To do that we first notice that
$$
\int_0^T\!\!\!\int_{\Omega_1} e^{2s\eta} \abs{f(y)}^2 dy
\leq \int_{\Omega_1} \abs{f(y)}^2 \para{
\int^T_0 e^{2s\eta(y,\tau^*)} d\tau } dy
= T\int_{\Omega_1} \abs{f(y)}^2 e^{2s\eta(y,\tau^*)}dy,
$$
since $\eta(y,\tau) \leq \eta(y, \tau^*)$ for all $(y,\tau) \in \Omega_1\times (0,T)$, as we have taken $T=2\tau^*$. Thus it follows from (\ref{5.3}) that
\begin{multline}\label{5.11}
s\int_{\Omega_1} \abs{f(y)}^2 e^{2s\eta(y,\tau^*)}dy
\leq Cs\int_{\Omega_1} \abs{A_2w(y,\tau^*)}^2 e^{2s\eta(y,\tau^*)}dy \\
+ Cs^3 \int_{0}^T\!\!\!\int_{\omega} e^{2s\eta} \ell^{-3}(\tau)
\abs{w_\tau}^2 dyd\tau+C\int_{0}^T\!\!\!\int_{\Omega_1} e^{2s\eta(y,\tau^*)} \abs{\mathcal{Q}_1w_\tau}^2 dyd\tau.
\end{multline}
Further, taking into account that $\sup_{y\in \Omega_1} e^{2s\eta(y,\tau^*)} < \infty$ and 
$\sup_{(y,\tau)\in \Omega_1\times(0,T)} \ell^{-3}(\tau) e^{2s\eta(y,\tau)} < \infty$,
(\ref{5.11}) then yields
\begin{multline}
s\int_{\Omega} \abs{f(y)}^2 e^{2s\eta(y,\tau^*)}dy
\leq Cs\norm{w(\cdot,\tau^*)}^2_{H^2(\Omega_1)} + Cs^3\int_{0}^T\!\!\!\int_{\omega} \abs{w_\tau}^2 dyd\tau\cr
+C\int_{0}^T\!\!\!\int_{\Omega_1\backslash \Omega_2} e^{2s\eta(y,\tau^*)} \para{\abs{w_\tau}^2+\abs{w_{\tau,y}}^2} dyd\tau,
\end{multline}
for $s$ sufficiently large. The next step involves noticing that
$$
\sup_{y\in\Omega_1\backslash\Omega_2}e^{2s\eta(y,\tau^*)}<\inf_{y\in\Omega}e^{2s\eta(y,\tau^*)}.
$$
Namely, putting
$$
m_1=\frac{e^{2\lambda\delta}-e^{2\lambda\overline{\psi}}}{\ell(\tau^*)}\ \textrm{and}\ m_2=\frac{e^{\lambda\delta}-e^{2\lambda\overline{\psi}}}{\ell(\tau^*)}\ \textrm{in such a way that}\ m=m_2-m_1<0,
$$
we deduce from the two assumptions (\ref{3.3})-(\ref{3.4}) that
$$
\eta(y,\tau^*)\geq m_1,\ y\in \Omega\ \textrm{and}\ \eta(y,\tau^*)\leq m_2,\ y\in \Omega_1\backslash \Omega_2.
$$
Therefore,
\begin{multline}
s\int_{\Omega} \abs{f(y)}^2dy
\leq C s\norm{w(\cdot,\tau^*)}^2_{H^2(\Omega_1)} + Cs^3\int_{0}^T\!\!\!\int_{\omega} \abs{w_\tau}^2 dyd\tau\cr
+Ce^{2ms}\int_{0}^T\!\!\!\int_{\Omega_1\backslash \Omega_2}\para{\abs{w_\tau}^2+\abs{w_{\tau,y}}^2} dyd\tau,
\end{multline}
for $s$ sufficiently large, whence
\begin{equation}\label{5.14}
s\norm{f}^2_{L^2(\Omega)}
\leq C s\norm{w(\cdot,\tau^*)}^2_{H^2(\Omega_1)} + C s\int_{0}^T\!\!\!\int_{\omega} \abs{w_\tau}^2 dyd\tau
+Ce^{2ms}\norm{f}^2_{L^2(\Omega)},
\end{equation}
by (\ref{4.9}). Now, taking $s > 0$ large enough in (\ref{5.14}), we end up getting
\begin{equation}\label{5.15}
\norm{f}^2_{L^2(\Omega)}
\leq C\norm{w(\cdot,\tau^*)}^2_{H^2(\Omega_1)} + \int_{0}^T\!\!\!\int_{\omega} \abs{w_\tau}^2 dyd\tau.
\end{equation}
The remaining part of the proof is to upper bound the last term in the right hand side of (\ref{5.15}). To do that we pick $\xi\in \mathcal{C}_0^\infty(\omega_1)$, where $\omega_1$ is the same as in (\ref{3.1}), such that $\xi(y)=1$ for $y \in \omega$. It is easy to check that $v=\xi w$ is solution to
$$
\left\{\begin{array}{lll}
\para{\p_\tau-\mathcal{L}_{a_1}}v=\mathcal{Q}_1(y,\p_y)w, & y\in \omega_1,\ \tau\in (0,T),\cr
v(y,0)=0, & y\in \omega_1,\cr
v(y,\tau)=0, & y\in\p \omega_1,\ \tau \in (0,T),
\end{array}
\right.
$$
where $\mathcal{Q}_1$ is a first order differential operator supported in $\omega_1\backslash \omega$. Therefore (\ref{4.8}) entails
\begin{equation}\label{5.17}
\norm{w_\tau}_{L^2((0,T)\times\omega)}\leq C\para{\norm{w}_{L^2((0,T)\times \omega_1)}+\norm{w_{y}}_{L^2((0,T)\times \omega_1)}}.
\end{equation}
By arguing as above with $\chi\in \mathcal{C}_0^\infty(\omega_2)$, where $\omega_2$ is defined in (\ref{3.1}), and such that $\chi(y)=1$ for $y \in \omega_1$, we find out that $z=\chi w$ is solution to following system
\begin{equation}\label{5.18}
\left\{\begin{array}{lll}
\para{\p_\tau-\mathcal{L}_a(y,\p)}z=\mathcal{Q}'_1(y,\p_y)w, & y\in \omega_2,\ \tau\in (0,T),\cr
z(y,0)=0, & y\in \omega_2,\cr
z(y,\tau)=0, & y\in\p \omega_2,\ \tau\in (0,T),
\end{array}
\right.
\end{equation}
where $\mathcal{Q}'_1$ is a first order differential operator supported in $\omega_2\backslash \omega_1$. Finally, multiplying the differential equation in (\ref{5.18}) by $\overline{z}$ and integrating by parts, yields
$$
\norm{w_{y}}_{L^2(\omega_1\times(0,T))}\leq C\norm{w}_{L^2(\omega_2\times(0,T))}.
$$
This, combined with (\ref{5.15}) and (\ref{5.17}), entails the desired result.
\end{proof}
\section{Completion of the proof of Theorem \ref{T1.1}}
The last step of the proof is to get rid of the last integral in the right hand side of (\ref{5.4}). Namely we shall prove that $\norm{w}_{L^2(\omega_2)\times(0,T)}$ can be majorized  by $\norm{w(\tau^*,\cdot)}_{H^2(\Omega_1)}$ up to some positive multiplicative constant.

Inspired by \cite{SU} we start by establishing the:
\begin{lemma}\label{L6.1}
Let $X$, $Y$ , $Z$ be three Banach spaces, let $\mathcal{A} : X \to Y$ be a bounded injective linear operator with domain $\mathscr{D}(\mathcal{A})$, and let $\mathcal{K}: X \to Z$ be a compact linear operator. Assume that there exists $C>0$ such that
\begin{equation}\label{6.1}
\norm{f}_X \leq C_1 \norm{\mathcal{A}f}_Y +\norm{\mathcal{K}f}_Z,\quad  \forall f \in \mathscr{D}(\mathcal{A}).
\end{equation}
Then there exists $C>0$ such that
\begin{equation}\label{6.2}
\norm{f}_X \leq C\norm{\mathcal{A}f}_Y,\quad \forall f \in \mathscr{D}(\mathcal{A}).
\end{equation}
\end{lemma}
\begin{proof}
Given $\mathcal{A}$ bounded and injective we argue by contradiction by assuming the opposite to (\ref{6.2}). Then there exists a sequence $(f_n)_n$ in $X$ such that $\norm{f_n}_X = 1$ for all $n$ and $\mathcal{A}f_n \to 0$ in $Y$ as $n$ go to infinity. Since $\mathcal{K}: X \to Z$ is compact, there is  a subsequence, still denoted by $f_n$, such that $\mathcal{K}f_n$ converges in $Z$. Therefore this is a Cauchy sequence in $Z$, hence, by applying (\ref{6.1}) to $f_n-f_m$, we get that $\norm{f_n-f_m}_X \to 0$, as $n\to\infty$, $m \to\infty$. As a consequence $(f_n)_n$ is a Cauchy sequence in $X$ so $f_n \to f$ as $n\to\infty$ for some $f \in X$. Since
$\norm{f_n}_X = 1$ for all $n$, we necessarily we have $\norm{f}_X=1$. Moreover it holds true that $\mathcal{A}f_n \to \mathcal{A}(f)= 0$ as $n \to \infty$, which is a contradiction to the fact that $\mathcal{A}$ is injective.
\end{proof}
Let us now introduce 
$$
X=\set{f\in L^2(\R),\ f(y)=0\ \textrm{in}\ \R\backslash \Omega},\ Y=H^2(\Omega_1),\ Z=L^2((0,T)\times \omega_2),
$$
and define
$$
\mathcal{A}: X\to Y,\ \mathcal{A}(f)=w(\tau^*,\cdot)\ \textrm{and}\ \mathcal{K} : X\to Z,\ \mathcal{K}(f)=w_{(0,T)\times \omega_2},
$$
where $w$ denotes the unique solution to (\ref{5.1}), so we can state the:
\begin{lemma}\label{L6.2}
The operator $\mathcal{A}$ is bounded and injective.
\end{lemma}
\begin{proof}
First the boundedness of $\mathcal{A}$ follows readily from (\ref{4.7}). Second $w$ being solution to
$$
\left\{
\begin{array}{llll}
\para{\p_\tau-\mathcal{L}_{a_1}(y,\p)}w=f(y)\alpha(y,\tau), & \textrm{in }\ \R\times(0,T),\cr
w(y,0)=0,& \textrm{in }\ \R,
\end{array}
\right.
$$
where $f=a_2-a_1$ and $\alpha=(w_{a_2})_{yy}-(w_{a_2})_{y}$,
we deduce from the identities $w(\cdot, \tau^*) = 0$ and $f = 0$ on $\omega_2\subset\Omega_1\backslash\Omega$ that $w_\tau(\cdot,\tau^*)=0$ on $\omega_2$. Arguing in the same way we get that the successive derivatives of $w$ wrt $\tau$ vanish on $\omega_2\times\set{\tau^*}$. Since $w$ is solution to some initial value problem with time independent coefficients, it is time analytic so we necessarily have $w= 0$ on $\omega_2\times(0,\tau^*)$. Therefore $f = 0$ on $\Omega$ by (\ref{5.4}), and the proof is complete.
\end{proof}
\begin{lemma}\label{L6.3}
$\mathcal{K}$ is a compact operator.
\end{lemma}
\begin{proof}
The operator $\mathcal{K}$ being bounded from $X$ to $H^1((0,T)\times\omega_2)$ as we have
$$
\norm{w}_{L^2((0,T);H^1(\omega_2))}+\norm{w_\tau}_{L^2((0,T);L^2(\omega_2))}\leq \norm{f}_{L^2(\Omega)},
$$
by (\ref{4.8}), the result follows readily from the compactness of the injection $H^1((0,T)\times\omega_2)\hookrightarrow Z=L^2((0,T)\times \omega_2)$.
\end{proof}
Finally, by putting Lemmas \ref{L5.2}, \ref{L6.1}, \ref{L6.2} and \ref{L6.3} together, we end up getting that 
$$
  \norm{a_1-a_2}_{L^2(\Omega)}=\norm{f}_{L^2(\Omega)}\leq C\norm{\mathcal{A}(f)}_Y=C\norm{w(\tau^*,\cdot)}_{H^2(\Omega_1)},
$$
which yields Theorem \ref{T1.1}.

\section{Appendix}

The existence of a weight function $\psi_0$ fulfilling the conditions prescribed by (\ref{3.2}) for some fixed subset $\omega$ of $\Omega_1$ can be checked from \cite{[FI]}. Nevertheless, for the sake of completeness and for the convenience of the reader, we give in this Appendix an explicit expression of such a function $\psi_0$ in the one-dimensional case examined in this article. 

To this purpose we set for all $a \in [1 \slash 2,1)$,
\begin{equation}
\label{a1}
f_a(x) = \mu x^n + (1-\mu) x,\ x \in (0,1),\ \textrm{where $n$ is taken so large that}\ a^n < 1 \slash 2\ \textrm{and}\ \mu=\frac{a-1 \slash 2}{a-a^n}.
\end{equation}
If $a \in (0,1 \slash 2)$, put
$$
f_a(x)=1-f_{1-a}(1-x),\ x \in (0,1),
$$
where $f_{1-a}$ is defined by (\ref{a1}). For every $a \in (0,1)$, it is not hard to check that the function
$$ \psi_0(x) = \sin(\pi f_a(x)),\ x \in (0,1), $$
obeys
$$ \psi_0(x) \geq 0,\ x \in (0,1)\ \textrm{and}\ |\psi_0'(x)|>0,\ x \in (0,1)\setminus \{ a\}. $$
In light of this, the general case of an open interval $\Omega_1=(\alpha,\beta)$, where $\alpha < \beta$ are two real numbers, can be handled by defining
\begin{equation}
\label{a3}
\psi_0(x) = \sin( \pi f_a(h^{-1}(x)) ),\ x \in \Omega_1,
\end{equation}
where $h^{-1}$ is the inverse function to
\begin{equation}
\label{a4}
\left\{
\begin{array}{lll}
h : & (0,1) \longrightarrow  \Omega_1\cr
 & y \longrightarrow y \alpha + (1-y) \beta.\cr
\end{array}
\right.
\end{equation}
Evidently, the function $\psi_0$ given by (\ref{a3})-(\ref{a4}) fulfills (\ref{3.2}) provided $h(a) \in \omega$.


\end{document}